\documentclass[12pt]{amsart}
\usepackage{amsmath,amsthm,amsfonts,amssymb,mathrsfs}
\date{\today}

\usepackage{color}
\input xymatrix
\xyoption{all}

\usepackage{hyperref}

\newtheorem{theorem}{Theorem}
\newtheorem*{theorem*}{Main Theorem}

\newtheorem{proposition}{Proposition}
\newtheorem{corollary}{Corollary}
\newtheorem{lemma}{Lemma}                 
\newtheorem*{question*}{Question}
\theoremstyle{definition}

\newcommand{\op}{\operatorname}

\begin{document}

\title[On locally compact semitopological graph inverse semigroups]{On locally compact semitopological graph inverse semigroups}

\author[S.~Bardyla]{Serhii~Bardyla}
\address{Faculty of Applied Mathematics and Informatics, National University of Lviv,
Universytetska 1, Lviv, 79000, Ukraine}
\email{sbardyla@yahoo.com}



\begin{abstract}
In this paper we investigate locally compact semitopological graph inverse semigroups.  Our main result is the following: if a directed graph $E$ is strongly connected and has finitely many vertices, then any Hausdorff shift-continuous locally compact topology on the graph inverse semigroup $G(E)$ is either compact or discrete. This result generalizes results of Gutik and Bardyla who proved the above dichotomy for Hausdorff locally compact shift-continuous topologies on polycyclic monoids $\mathcal{P}_1$ and $\mathcal{P}_{\lambda}$, respectively.
\end{abstract}
\maketitle

\section{Introduction and Background}
In this paper all topological spaces are assumed to be Hausdorff. We shall follow the terminology of~\cite{Clifford-Preston-1961-1967,
Engelking-1989, Lawson-1998, Ruppert-1984}.
A semigroup $S$ is called an \emph{inverse semigroup} if for each element $a\in S$ there exists a unique
element $a^{-1}\in S$ such that $aa^{-1}a=a$ and $a^{-1}aa^{-1}=a^{-1}$. The element $a^{-1}$ is called the {\em inverse} of $a$.
The map $S\to S$, $x\mapsto x^{-1}$ assigning to each element of an inverse semigroup its inverse is called the \emph{inversion}.

A {\em directed graph} $E=(E^{0},E^{1},r,s)$ consists of sets $E^{0},E^{1}$ of {\em vertices} and {\em edges}, respectively, together with functions $s,r:E^{1}\rightarrow E^{0}$, called the {\em source} and the {\em range} functions, respectively. In this paper we refer to directed graphs simply as ``graphs".  A {\em path} $x=e_{1}\ldots e_{n}$ in a graph $E$ is a finite sequence of edges $e_{1},\ldots,e_{n}$ such that $r(e_{i})=s(e_{i+1})$ for each positive integer $i<n$.  We extend the source and range functions $s$ and $r$ on the set $\operatorname{Path}(E)$ of all pathes in graph $E$ as follows: for each $x=e_{1}\ldots e_{n}\in \operatorname{Path}(E)$ put $s(x)=s(e_{1})$ and $r(x)=r(e_{n})$. By $|x|$ we denote the length of the path $x$. We consider each vertex being a path of length zero. An edge $e$ is called a {\em loop} if $s(e)=r(e)$. A path $x$ is called a {\em cycle} if $s(x)=r(x)$ and $|x|>0$.
Let $a=e_1\ldots e_n$ and $b=f_1\ldots f_m$ be two paths such that $r(a)=s(b)$. Then by $ab$ we denote the path $e_1\ldots e_nf_1\ldots f_m$.
A path $x$ is called a {\em prefix} (resp. {\em suffix}) of a path $y$ if there exists path $z$ such that $y=xz$ (resp. $y=zx$).
A graph $E$ is called {\em finite} if the sets $E^0$ and $E^1$ are finite and {\em infinite} in the other case. A graph $E$ is called {\em strongly connected} if for each pair of vertices $e,f\in E^0$ there exist paths $u,v\in \op{Path}(E)$ such that $s(u)=r(v)=e$ and $s(v)=r(u)=f$.

A topological (inverse) semigroup is a Hausdorff topological space together with a continuous semigroup operation (and an~inversion, respectively).  If $S$ is a semigroup (an inverse semigroup) and $\tau$ is a topology on $S$ such that $(S,\tau)$ is a topological (inverse) semigroup, then we
shall call $\tau$ a (\emph{inverse}) \emph{semigroup} \emph{topology} on $S$.
A semitopological semigroup is a Hausdorff topological space together with a separately continuous semigroup operation. For each element $x$ of a semigroup $S$ the map $l_x(s):s\rightarrow xs$ ($r_x(s):s\rightarrow sx$, resp.) is called a left (right, resp.) shift on the element $x$. Observe that semigroup $S$ endowed with a topology is semitopological iff for each element $x\in S$ left and right shifts are continuous.
A topology $\tau$ on a semigroup $S$ is called shift-continuous if $(S,\tau)$ is a semitopological semigroup.
A semitopological inverse semigroup $S$ is called quasi-topological if the inversion map $S\to S$, $x\mapsto x^{-1}$, is continuous.

The bicyclic monoid ${\mathscr{C}}(p,q)$ is the semigroup with the identity $1$ generated by two elements $p$ and $q$ subject to the condition $pq=1$.
The bicyclic semigroup admits only the discrete semigroup topology~\cite{Eberhart-Selden-1969}. In \cite{Bertman-West-1976} this result was extended over the case of semitopological semigroups. The closure of a bicyclic semigroup in a locally compact topological inverse semigroup was described  in~\cite{Eberhart-Selden-1969}.
In~\cite{Gutik-2015} Gutik proved the following theorem.

\begin{theorem}[{\cite[Theorem 1]{Gutik-2015}}]\label{Gutik}
Any locally compact shift-continuous topology on the bicyclic monoid with adjoined zero is either compact or discrete.
\end{theorem}

In~\cite{Bardyla-2017(3)} Gutik's Theorem was generalized over the $\alpha$-bicyclic monoid.

One of generalizations of the bicyclic semigroup is a $\lambda$-polycyclic monoid.
For a non-zero cardinal $\lambda$, the $\lambda$-polycyclic monoid $\mathcal{P}_\lambda$ is the semigroup with identity and zero given by the presentation:
\begin{equation*}
    \mathcal{P}_\lambda=\left\langle \left\{p_i\right\}_{i\in\lambda}, \left\{p_i^{-1}\right\}_{i\in\lambda}\mid  p_i^{-1}p_i=1, p_j^{-1}p_i=0 \hbox{~for~} i\neq j\right\rangle.
\end{equation*}

Polycyclic monoid $\mathcal{P}_{k}$ over a finite cardinal $k$ was introduced in \cite{Nivat-Perrot-1970}. Algebraic properties of a semigroup $\mathcal{P}_{k}$ were investigated in \cite{Lawson-2009} and \cite{Meakin-1993}. Algebraic and topological properties of the $\lambda$-polycyclic monoid were investigated in~\cite{BardGut-2016(1)} and~\cite{BardGut-2016(2)}. In particular, it was proved that for every non-zero cardinal $\lambda$ the only locally compact semigroup topology on the $\lambda$-polycyclic monoid is the discrete topology. Observe that the bicyclic semigroup with an adjoined zero is isomorphic to the polycyclic monoid $\mathcal{P}_{1}$. Hence Gutik's Theorem~\ref{Gutik} can be reformulated in the following way: any locally compact shift-continuous topology on the polycyclic monoid $\mathcal{P}_{1}$ is either compact or discrete. In \cite{Bardyla-2016(1)} Theorem~\ref{Gutik} was generalized as follows.

 \begin{theorem}[{\cite[Main Theorem]{Bardyla-2016(1)}}]\label{Bard}
 Any locally compact shift-continuous topology on the $\lambda$-polycyclic monoid $\mathcal{P}_{\lambda}$ is either compact or discrete.
 \end{theorem}


For a directed graph $E=(E^{0},E^{1},r,s)$ the graph inverse semigroup (or simply GIS) $G(E)$ over $E$ is a semigroup with zero generated by the sets $E^{0}$, $E^{1}$ together with the set $E^{-1}=\{e^{-1}\mid e\in E^{1}\}$ satisfying the following relations for all $a,b\in E^{0}$ and $e,f\in E^{1}$:
 \begin{itemize}
 \item [(i)]  $a\cdot b=a$ if $a=b$ and $a\cdot b=0$ if $a\neq b$;
 \item [(ii)] $s(e)\cdot e=e\cdot r(e)=e;$
 \item [(iii)] $e^{-1}\cdot s(e)=r(e)\cdot e^{-1}=e^{-1};$
 \item [(iv)] $e^{-1}\cdot f=r(e)$ if $e=f$ and $e^{-1}\cdot f=0$ if $e\neq f$.
\end{itemize}

Graph inverse semigroups are generalizations of the polycyclic monoids. In particular, for every non-zero cardinal $\lambda$, the $\lambda$-polycyclic monoid is isomorphic to the graph inverse semigroup over the graph $E$ which consists of one vertex and $\lambda$ distinct loops.
However, in~\cite{Bardyla-2017(2)} it was proved that the $\lambda$-polycyclic monoid is a universal object in the class of graph inverse semigroups. More precisely, each GIS $G(E)$ embeds as an inverse subsemigroup into the $\lambda$-polycyclic monoid $\mathcal{P}_{\lambda}$ with  $\lambda\ge |G(E)|$.

According to~\cite[Chapter~3.1]{Jones-2011}, each non-zero element of the graph inverse semigroup $G(E)$ is of the form  $uv^{-1}$ where $u,v\in \operatorname{Path}(E)\hbox{ and } r(u)=r(v)$. A semigroup operation in $G(E)$ is defined by the formulas:
$$  u_1v_1^{-1}\cdot u_2v_2^{-1}=
    \begin{cases}
        u_1wv_2^{-1}, & \mbox{if $u_2=v_1w$ for some $w\in \operatorname{Path}(E)$};\\
        u_1(v_2w)^{-1},   & \mbox{if $v_1=u_2w$ for some $w\in \operatorname{Path}(E)$};\\
        0,              & \mbox{otherwise},
      \end{cases}$$
 and $$uv^{-1}\cdot 0=0\cdot uv^{-1}=0\cdot 0=0.$$

Simple verifications show that $G(E)$ is an inverse semigroup and $(uv^{-1})^{-1}=vu^{-1}$.

We shall say that GIS $G(E)$ satisfies condition $(\star)$ if for each infinite subset $A\subset \operatorname{Path}(E)$ there exists an infinite subset $B\subset A$ and an element $\mu\in G(E)$ such that for each $x\in B$, $\mu\cdot x\in \operatorname{Path}(E)$ and $|\mu\cdot x|>|x|$.

Graph inverse semigroups play an important role in the study of rings and $C^{*}$-algebras (see \cite{Abrams-2005,Ara-2007,Cuntz-1980,Kumjian-1998,Paterson-1999}).
Algebraic properties of graph inverse semigroups were studied in \cite{Amal-2016, Bardyla-2017(2), Jones-2011, Jones-Lawson-2014, Lawson-2009, Mesyan-2016}. In~\cite{Mesyan-Mitchell-Morayne-Peresse-2013} it was showed that a locally compact topological GIS $G(E)$ over a finite graph $E$ is discrete. In~\cite[Theorem 1]{Bardyla-2017(1)} the author characterized graph inverse semigroups admitting only discrete locally compact semigroup topology:
\begin{theorem}\label{Bard1}
  The discrete topology is the only locally compact semigroup topology on a graph inverse semigroup $G(E)$ if and only if $G(E)$ satisfies the condition~$(\star)$.
\end{theorem}

Further we shall often use the following fact proved in~\cite[Lemma 1]{Mesyan-Mitchell-Morayne-Peresse-2013}:
\begin{lemma}\label{lemmaeq}
For any $a,b\in G(E)\setminus\{0\}$, the sets $\{x\in G(E)\mid x\cdot a=b\}$ and $\{x\in G(E)\mid a\cdot x=b\}$ are finite.
\end{lemma}

\section{Main results}
Let $G(E)$ be the graph inverse semigroup over a graph $E$. Fix an arbitrary vertex $e\in E^0$ and let $C^e:=\{u\in \op{Path}(E)\mid s(u)=r(u)=e\}$.
Put $$C_1^e:=\{u\in C^e\mid r(v)\neq e\hbox{ for each non-trivial prefix } v \hbox{ of } u\}.$$
By $\langle C^e\rangle$ (resp. $\langle C_1^e\rangle$) we denote the inverse subsemigroup of $G(E)$ which is generated by the set $C^e$
(resp. $C_1^e$). Observe that $e\in C_1^e$ and $e$ is the identity in $\langle C^e\rangle$.

\begin{lemma}\label{lemma0}
For each vertex $e\in E^0$ of an arbitrary graph $E$ the following statements hold:
\begin{itemize}
\item[$1)$] if $C_1^e=\{e\}$ then $\langle C^e\rangle=\{e\}$;
\item[$2)$] if $|C_1^e\setminus\{e\}|=1$ then $\langle C^e\rangle$ is isomorphic to the bicyclic monoid;
\item[$3)$] if $|C_1^e\setminus\{e\}|=\lambda>1$ then  $\langle C^e\rangle$ is isomorphic to the $\lambda$-polycyclic monoid $\mathcal{P}_{\lambda}$.
\end{itemize}
\end{lemma}

\begin{proof}
Fix an arbitrary vertex $e\in E^0$.
The statement $1$ is obvious.

Now we prove the statement $3$. Suppose that $|C_1^e\setminus\{e\}|=\lambda>1$. Let $C_1^e\setminus\{e\}=\{u_{\alpha}\}_{\alpha\in \lambda}$ be an enumeration of $C_1^e\setminus\{e\}$. For convenience we put $e=u_{-1}$.
Observe that for each element $v\in C^e$ there exist elements $u_{\alpha_1}, u_{\alpha_2},\ldots, u_{\alpha_n}\in C_1^e$ such that $v=u_{\alpha_1}u_{\alpha_2}\ldots u_{\alpha_n}$. Simple verifications show that
$$\langle C_1^e\rangle=\{uv^{-1}\mid u,v\in C^e\}\cup\{0\}=\langle C^e\rangle.$$

Let $G=\{p_{\alpha}\}_{\alpha\in \lambda}\cup\{p_{\alpha}^{-1}\}_{\alpha\in \lambda}$ be the set of generators of $\mathcal{P}_{\lambda}$.
We define a map $f:C_1^e\rightarrow \mathcal{P}_{\lambda}$ in the following way:
$f(u_{-1})=1$ and $f(u_{\alpha})=p_{\alpha}$ for each $\alpha\in\lambda$. Extend the map $f$ on the set $\langle C^e\rangle$ in the following way:  for each element $u=u_{\alpha_1}u_{\alpha_2}\ldots u_{\alpha_n}\in C^e$ put $f(u)=p_{\alpha_{1}}p_{\alpha_{2}}\ldots p_{\alpha_{n}}$.
For each non-zero element $uv^{-1}\in \langle C^e\rangle$ put $f(uv^{-1})=f(u)f(v)^{-1}$ and $f(0)=0$. Obviously, $f$ is a bijection. Let us show that $f$ is a homomorphism. Fix  arbitrary elements $ab^{-1},cd^{-1}\in \langle C^e\rangle$, where
$$a=u_{\alpha_1}\ldots u_{\alpha_n}, b=u_{\beta_1}\ldots u_{\beta_{m}}, c=u_{\gamma_1},\ldots u_{\gamma_k}, d=u_{\delta_1}\ldots u_{\delta_t}.$$ There are three cases to consider:
\begin{itemize}
\item[$(1)$] $ab^{-1}\cdot cd^{-1}=ac_1d^{-1}$, i.e., $c=bc_1$;
\item[$(2)$] $ab^{-1}\cdot cd^{-1}=a(db_1)^{-1}$, i.e., $b=cb_1$;
\item[$(3)$] $ab^{-1}\cdot cd^{-1}=0$.
\end{itemize}
Suppose that case $(1)$ holds, i.e., $u_{\gamma_1},\ldots u_{\gamma_k}=u_{\beta_1}\ldots u_{\beta_{m}}u_{\gamma_{m+1}}\ldots u_{\gamma_{k}}$. Observe that
\begin{equation*}
f(ac_1d^{-1})=f(u_{\alpha_1}\ldots u_{\alpha_n}u_{\gamma_{m+1}}\ldots u_{\gamma_{k}})f(u_{\delta_1}\ldots u_{\delta_t})^{-1}=p_{\alpha_1}\ldots p_{\alpha_n}p_{\gamma_{m+1}}\ldots p_{\gamma_{k}}(p_{\delta_1}\ldots p_{\delta_t})^{-1}.
\end{equation*}
On the other hand
\begin{equation*}
\begin{split}
&f(ab^{-1})\cdot f(cd^{-1})=p_{\alpha_1}\ldots p_{\alpha_n}(p_{\beta_1}\ldots p_{\beta_m})^{-1}\cdot p_{\beta_1}\ldots p_{\beta_{m}}p_{\gamma_{m+1}}\ldots p_{\gamma_{k}}\cdot (p_{\delta_1}\ldots p_{\delta_t})^{-1}=\\
&= p_{\alpha_1}\ldots p_{\alpha_n}p_{\gamma_{m+1}}\ldots p_{\gamma_{k}}(p_{\delta_1}\ldots p_{\delta_t})^{-1}=f(ac_1d^{-1}).
\end{split}
\end{equation*}

Case $(2)$ is similar to case $(1)$. Consider case $(3)$. In this case there exists a positive integer $i$ such that $u_{\beta_j}=u_{\gamma_j}$ for every $j<i$ and $u_{\beta_i}\neq u_{\gamma_i}$. Observe that $f(ab^{-1}\cdot cd^{-1})=f(0)=0$.
\begin{equation*}
\begin{split}
&f(ab^{-1})\cdot f(cd^{-1})= p_{\alpha_1}\ldots p_{\alpha_n}p_{\beta_m}^{-1}\ldots p_{\beta_i}^{-1}(p_{\beta_{i-1}}^{-1}\ldots p_{\beta_{1}}^{-1}\cdot p_{\beta_1}\ldots p_{\beta_{i-1}})p_{\gamma_{i}}\ldots p_{\gamma_{k}}\cdot (p_{\delta_1}\ldots p_{\delta_t})^{-1}=\\
&=p_{\alpha_1}\ldots p_{\alpha_n}p_{\beta_m}^{-1}\ldots (p_{\beta_i}^{-1}\cdot p_{\gamma_{i}})\ldots p_{\gamma_{k}}\cdot (p_{\delta_1}\ldots p_{\delta_t})^{-1}=0=f(ab^{-1}\cdot cd^{-1}).\\
\end{split}
\end{equation*}
Hence map $f$ is an isomorphism.

Proof of statement $2$ is similar to that of the statement $3$.
\end{proof}

The following Theorem extends Theorem~3 from \cite{Mesyan-Mitchell-Morayne-Peresse-2013} and Proposition~3.1 from \cite{BardGut-2016(1)} over the case of semitopological graph inverse semigroups.
\begin{theorem}\label{lemma1}
Let $G(E)$ be a semitopological GIS.
Then each non-zero element of $G(E)$ is an isolated point in $G(E)$.
\end{theorem}

\begin{proof}
First we prove that each vertex $a$ of the graph $E$ is an isolated point in $G(E)$.
There are two cases to consider:
\begin{itemize}
\item[$1)$] there exists an edge $x$ such that $s(x)=a$;
\item[$2)$] the set $\{x\in E^1\mid s(x)=a\}$ is empty.
\end{itemize}
First consider the case $1$. Fix an arbitrary edge $x$ such that $s(x)=a$. Observe that both sets $xx^{-1}\cdot G(E)$ and $G(E)\cdot xx^{-1}$ are retracts of $G(E)$ and do not contain point $a$. Then $U(a)=G(E)\setminus (xx^{-1}G(E)\cup G(E)\cdot xx^{-1})$ is an open neighborhood of $a$. Fix an arbitrary open neighborhood $U(xx^{-1})$ which does not contain $0$. Since $xx^{-1}\cdot a\cdot xx^{-1}=xx^{-1}$ the continuity of left and right shifts in $G(E)$ yields an open neighborhood $V(a)\subset U(a)$ such that $xx^{-1}\cdot V(a)\cdot xx^{-1}\subset U(xx^{-1})$. Fix an arbitrary element $bc^{-1}\in V(a)$. Observe that the choice of $U(a)$ implies that $x$ is neither a prefix of $b$ nor $c$ (in the other case $bc^{-1}=xx^{-1}\cdot bc^{-1}\in xx^{-1}\cdot G(E)$ or $bc^{-1}=bc^{-1}\cdot xx^{-1} \in G(E)\cdot xx^{-1}$). Since the set $U(xx^{-1})$ does not contain $0$ we obtain that $xx^{-1}\cdot bc^{-1}\cdot xx^{-1}\neq 0$ and, as a consequence, $b$ and $c$ are prefixes of $x$. Hence  $b=c=a$ which implies that $V(a)=\{a\}$.

Next consider the case $2$. Since $a\cdot a\cdot a=a$, the continuity of left and right shifts in $G(E)$ yields an open neighborhood $V(a)$ such that $a\cdot V(a)\cdot a\subset G(E)\setminus\{0\}$. Fix an arbitrary element $bc^{-1}\in V(a)$. Since $s(b)\neq a$ and $s(c)\neq a$ we obtain that $a\cdot bc^{-1}\cdot a\neq 0$ iff $b=c=a$ which implies that $V(a)=\{a\}$.

Hence each vertex $a$ is an isolated point in $G(E)$.
Fix an arbitrary non-zero element $uv^{-1}\in G(E)$. Since $u^{-1}\cdot uv^{-1}\cdot v=v^{-1}\cdot v=r(v)$, the continuity of left and right shifts in $G(E)$ yields an open neighborhood $V$ of $uv^{-1}$ such that $u^{-1}\cdot V\cdot v\subseteq\{r(v)\}$. By Lemma~\ref{lemmaeq}, the set $u^{-1}\cdot V$ is finite. Repeating our arguments, by Lemma~\ref{lemmaeq}, the set $V$ is finite which implies that point $uv^{-1}$ is isolated in $G(E)$.
\end{proof}

Theorem~\ref{lemma1} implies the following:
\begin{corollary}\label{corol1}
Let $G(E)$ be a locally compact non-discrete semitopological GIS. Then for each compact neighborhoods $U,V$ of $0$ the set $U\setminus V$ is finite.
\end{corollary}

\begin{lemma}\label{lemmacomp}
Each infinite GIS $G(E)$ admits a unique compact non-discrete shift-continuous topology $\tau$. Moreover, the inversion is continuous in $(G(E),\tau)$.
\end{lemma}
\begin{proof}
The topology $\tau$ is defined in the following way: each non-zero element is isolated in $(G(E),\tau)$ and an open neighborhood base of $0$ consists of cofinite subsets of $G(E)$ which contain $0$.
Since for each open neighborhood $V$ of $0$, the set $V^{-1}$ is cofinite in $G(E)$ and contains $0$ we obtain that the inversion is continuous in $(G(E),\tau)$. To prove the continuity of left and right shifts in $(G(E),\tau)$ we need to check it at the unique non-isolated point $0$.
Fix an arbitrary non-zero element $uv^{-1}\in G(E)$ and an open neighborhood $U$ of $0$. By the definition of topology $\tau$ the set $A=G(E)\setminus U$ is finite. By Lemma~\ref{lemmaeq}, the set $B=\{ab^{-1}\in G(E)\mid uv^{-1}\cdot ab^{-1}\in A\}$ is finite and, obviously, does not contain $0$. Then $V=G(E)\setminus B$ is an open neighborhood of $0$ such that $uv^{-1}\cdot V\subseteq U$. Hence left shifts are continuous in $(G(E),\tau)$. Continuity of right shifts in $G(E)$ can be proved similarly.
\end{proof}

Let $G(E)$ be an arbitrary GIS and $\mathcal{L}, \mathcal{R}, \mathcal{D}$ be the Green relations on $G(E)$. By Lemma~3.1.13 from~\cite{Jones-2011} for any two non-zero elements $ab^{-1}$ and $cd^{-1}$ of $G(E)$ the following conditions hold:
\begin{itemize}
\item[(1)] $ab^{-1}\mathcal{L} cd^{-1}$ iff $b=d$;
\item[(2)] $ab^{-1}\mathcal{R} cd^{-1}$ iff $a=c$;
\item[(3)] $ab^{-1}\mathcal{D} cd^{-1}$ iff $r(a)=r(b)=r(c)=r(d)$.
\end{itemize}

Further, for a path $u\in \op{Path}(E)$ by $L_u$ (resp. $R_u$) we denote an $\mathcal{L}$-class (resp. $\mathcal{R}$-class) which contains the element $uu^{-1}$. For a vertex $e\in E^0$ by $D_e$ denote the $\mathcal{D}$-class containing $e$. The condition $(3)$ implies that each non-zero $\mathcal{D}$-class contains exactly one vertex.

Recall that GIS $G(E)$ satisfies the condition $(\star)$ if for each infinite subset $A\subset \operatorname{Path}(E)$ there exists an infinite subset $B\subset A$ and an element $\mu\in G(E)$ such that for each $x\in B$, $\mu\cdot x\in \operatorname{Path}(E)$ and $|\mu\cdot x|>|x|$.

\begin{lemma}\label{lemma2}
Let $G(E)$ be a locally compact non-discrete semitopological GIS satisfying the condition~$(\star)$.
Then there exists an element $v\in\op{Path}(E)$ such that for each open compact neighborhood $U$ of $0$ the set $L_v\cap U$ is infinite.
\end{lemma}

\begin{proof}
To derive a contradiction, suppose  that for each element $v\in\op{Path}(E)$ there exists an open compact neighborhood $W_v$ of $0$ such that the set $L_v\cap W_v$ is finite.  Fix an arbitrary open compact neighborhood $U$ of $0$. By Corollary~\ref{corol1}, the set $U\setminus W_v$ is finite for each element $v\in\op{Path}(E)$. Hence the set $U\cap L_v$ is finite for each path $v$.  Let $T=\{v\in\op{Path}(E)\mid L_v\cap U\neq\emptyset\}$. Since the set $U$ is infinite we obtain that the set $T$ is infinite as well.
For each $v\in T$ fix an element $u_vv^{-1}\in L_v\cap U$ such that $|u_v|\geq |y|$ for every element $yv^{-1}\in L_v\cap U$.
Since $G(E)$ satisfies the condition $(\star)$, there exists an infinite subset $A\subset \{u_v\}_{v\in T}$ and an element $\mu\in G(E)$ such that $\mu\cdot y\in \op{Path}(E)$ and $|\mu\cdot y|>|y|$ for each element $y\in A$. Since $\mu\cdot 0=0$, the continuity of left shifts in $G(E)$ yields an open neighborhood $V$ of $0$ such that $\mu\cdot V\subset U$. Since the set $U\setminus V$ is finite (see Corollary~\ref{corol1}), we obtain that there exists an element $v\in T$ such that $u_vv^{-1}\in V\cap U$. Observe that $\mu\cdot u_vv^{-1}\neq 0$, because $\mu\cdot u_v\in\op{Path}(E)$ and $r(\mu\cdot u_v)=r(u_v)=r(v)$. Hence $\mu\cdot u_vv^{-1}\in L_v\cap U$ and $|\mu\cdot u_v|>|u_v|$ which contradicts the choice of the element $u_{v}v^{-1}$.
\end{proof}

\begin{lemma}\label{lemma3}
Let $G(E)$ be a locally compact non-discrete semitopological GIS satisfying the condition~$(\star)$.
Then there exists a $\mathcal{D}$-class $D_e$ such that the set $L\cap U$ is infinite for each open neighborhood $U$ of $0$ and $\mathcal{L}$-class $L\subset D_e$.
\end{lemma}

\begin{proof}
By Lemma~\ref{lemma2}, there exists element $v\in\op{Path}(E)$ such that the set $L_v\cap U$ is infinite for each open compact neighborhood $U$ of $0$. Recall that $D_{r(v)}=\{ab^{-1}\mid r(a)=r(b)=r(v)\}$. Fix an arbitrary element $u\in\op{Path}(E)\cap D_{r(v)}$ and an open compact neighborhood $U$ of $0$. Observe that element $vu^{-1}\neq 0$, because $r(u)=r(v)$. Since $0\cdot vu^{-1}=0$ the continuity of right shifts in $G(E)$ yields an open neighborhood $V$ of $0$ such that $V\cdot vu^{-1}\subset U$. Observe that $L_v\cdot vu^{-1}=L_u$. By Corollary~\ref{corol1}, the set $L_v\cap V$ is infinite. By Lemma~\ref{lemmaeq}, $(L_v\cap V)\cdot vu^{-1}$ is an infinite subset of $U\cap L_u$.
\end{proof}

Now our aim is to prove our main result which generalizes Theorem~\ref{Gutik} and Theorem~\ref{Bard}.
\begin{theorem*}\label{main}
Let $E$ be a strongly connected graph which has finitely many vertices.
Then any locally compact shift-continuous topology on GIS $G(E)$ is either compact or discrete.
\end{theorem*}

\section*{Proof of Main Theorem}
The proof of Main Theorem is divided into a series of $5$ lemmas.
In the following lemmas~\ref{lemma4}--\ref{lemma8} we assume that graph $E$ is strongly connected and has finitely many vertices.  As a consequence, the semigroup $G(E)$ satisfies the condition $(\star)$ (see Remark~2 from \cite{Bardyla-2017(1)}).
By Theorem~\ref{Bard}, Main Theorem holds if the graph $E$ contains only one vertex (in this case $G(E)$ is either finite or isomorphic to a $\lambda$-polycyclic monoid). Hence we can assume that the graph $E$ contains at least two vertices.
By $e$ we denote an arbitrary vertex such that the set $L\cap U$ is infinite for any open neighborhood $U$ of $0$ and any $\mathcal{L}$-class $L\subset D_e$ (Lemma~\ref{lemma3} implies that such vertex $e$ exists).
Recall that by $\langle C^e\rangle$ we denote the inverse subsemigroup of $G(E)$ which is generated by the set $C^e=\{u\in \op{Path}(E)\mid d(u)=r(u)=e\}$.

\begin{lemma}\label{lemma4}
Let $G(E)$ be a locally compact non-discrete semitopological GIS.
Then the set $\langle C^e\rangle\setminus U$ is finite for each open neighborhood $U$ of $0$.
\end{lemma}

\begin{proof}
By the assumption, there exists a vertex $f$ and paths $x,y$ such that $s(x)=r(y)=e$ and $r(x)=s(y)=f$. Since $xy\in C_1^e$, by Lemma~\ref{lemma0}, $\langle C^e\rangle$ is an infinite set.
Fix an arbitrary compact open neighborhood $U$ of $0$. Recall that $L_e\cap U$ is infinite. Since the graph $E$ contains finitely many vertices, there exists a vertex $f$ such that the set $B=\{u\in L_e\cap U\mid s(u)=f\}$ is infinite. We claim that $0$ is a limit point of $\langle C^e\rangle$.
Indeed, if $f=e$ then $B\subset \langle C^e\rangle$ and hence $0$ is a limit point of $\langle C^e\rangle$. Assume that $f\neq e$. Since graph $E$ is strongly connected, there exists a path $v\in\op{Path}(E)$ such that $s(v)=e$ and $r(v)=f$. Since $v\cdot 0=0$, the continuity of right shifts in $G(E)$ yields an open neighborhood $V$ of $0$ such that $v\cdot V\subset U$. By Corollary~\ref{corol1}, the set $U\setminus V$ is finite  which implies that the set $B\cap V$ is infinite. By Lemma~\ref{lemmaeq}, the set $v(V\cap B)$ is an infinite subset of $U$. Observe that for each element $u\in vB$, $s(u)=r(u)=e$. Hence $0$ is a limit point of $\langle C^e\rangle$. Observe that $\langle C^e\rangle\cup\{0\}$ is a closed and hence locally compact subsemigroup of $G(E)$ which is isomorphic to the polycyclic monoid $\mathcal{P}_{\lambda}$ where $\lambda=|C_1^e\setminus\{e\}|$ (see Lemma~\ref{lemma0}). By Theorem~\ref{Bard}, semigroup $\langle C^e\rangle$ is compact which implies that $\langle C^e\rangle\setminus U$ is finite for each open neighborhood $U$ of $0$.
\end{proof}

\begin{lemma}\label{lemma5}
Let $G(E)$ be a locally compact non-discrete semitopological GIS.
Then the set $L_e\setminus U$ is finite for each open neighborhood $U$ of $0$.
\end{lemma}

\begin{proof}
Suppose that there exists an open compact neighborhood $U$ of $0$ such that the set $A=L_e\setminus U$ is infinite.  Since the graph $E$ has finitely many vertices, we can find a vertex $f$ and an infinite subset $B\subset A$ such that $s(u)=f$ for each element $u\in B$. The strong connectedness of the graph $E$ yields a path $v$ such that $s(v)=e$ and $r(v)=f$. Observe that Lemma~\ref{lemma4} implies that the set $vB\cap U$ is infinite, because $vB$ is an infinite subset of $\langle C^e\rangle$.
Since $v^{-1}\cdot 0=0$, the continuity of left shifts in $G(E)$ yields an open neighborhood $V$ of $0$ such that $v^{-1}\cdot V\subset U$. By Corollary~\ref{corol1}, the set $U\setminus V$ is finite. Then there exists an element $b\in B$ such that $vb\in V$. Hence $v^{-1}\cdot vb=b\in U$ which contradicts the choice of $U$.
\end{proof}

\begin{lemma}\label{lemma6}
Let $G(E)$ be a locally compact non-discrete semitopological GIS.
Then the set $L\setminus U$ is finite for any open neighborhood $U$ of $0$ and any $\mathcal{L}$-class $L\subset D_e$.
\end{lemma}

\begin{proof}
Fix an arbitrary $\mathcal{L}$-class $L\subset D_e$ and an open compact neighborhood $U$ of $0$. Clearly, $L=L_v$ for some path $v$ such that $r(v)=e$.
Since $0\cdot v^{-1}=0$, the continuity of right shifts in $G(E)$ yields an open neighborhood $V$ of $0$ such that $V\cdot v^{-1}\subset U$.
Observe that $L_e\cdot v^{-1}=L_v$. By Lemma~\ref{lemma5}, the set $L_e\setminus V$ is finite. Hence the set $L_v\setminus U$ is finite as well.
\end{proof}

\begin{lemma}\label{lemma7}
Let $G(E)$ be a locally compact non-discrete semitopological GIS.
Then the set $D_e\setminus U$ is finite for each open neighborhood $U$ of $0$.
\end{lemma}

\begin{proof} To derive a contradiction, suppose that there exists an open neighborhood $U$ of $0$ such that the set $A=D_e\setminus U$ is infinite. Without loss of generality we can assume that $U$ is compact. Put $T=\{v\in \op{Path}(E)\cap D_e\mid L_v\setminus U\neq\emptyset\}$. By Lemma~\ref{lemma6}, the set $L_v\setminus U$ is finite for each path $v\in D_e$. Since the set $U$ is infinite, we obtain that the set $T$ is infinite as well. For each path $v\in T$ by $u_v$ we denote an arbitrary path satisfying the following conditions:
\begin{itemize}
\item[$\bullet$] $u_vv^{-1}\notin U$;
\item[$\bullet$] if $uv^{-1}\notin U$ for some path $u$ then $|u_v|\geq |u|$.
\end{itemize}
Since the set $T$ is infinite, the set $B=\{u_vv^{-1}|\hbox{ }v\in T\}$ is infinite as well.
Since graph $E$ has finitely many vertices, there exists a vertex $f$ and an infinite subset $C\subset T$ such that $s(u_{v})=f$ for each path $v\in C$. Put $D=\{u_vv^{-1}\in B|\hbox{ } v\in C\}$. Fix an arbitrary path $a$ such that $|a|\geq 1$ and $r(a)=f$ (by the strong connectedness of graph $E$ such path $a$ always exists). The choice of elements $u_v$ implies that $aD=\{au_{v}v^{-1}|\hbox{ }v\in C\}$ is an infinite subset of $U$. Since $a^{-1}\cdot 0=0$, the continuity of left shifts in $G(E)$ yields an open neighborhood $V\subset U$ such that $a^{-1}\cdot V\subset U$.
By Corollary~\ref{corol1}, the set $U\setminus V$ is finite which implies that the set $V\cap aD$ is infinite. Fix an arbitrary element $au_vv^{-1}\in aD\cap V$. Then $a^{-1}\cdot au_vv^{-1}=u_vv^{-1}\in U\cap D$ which contradicts the choice of $U$.
\end{proof}

The following lemma completes the proof of Main Theorem.
\begin{lemma}\label{lemma8}
Any non-discrete locally compact shift-continuous topology on GIS $G(E)$ is compact.
\end{lemma}

\begin{proof}
By Lemma~\ref{lemma5}, the set $L_e\setminus U$ is finite for each open neighborhood $U$ of $0$.
Fix an arbitrary compact open neighborhood $U$ of $0$ and an arbitrary vertex $f\in E^0\setminus\{e\}$. The strong connectedness of the graph $E$ implies that there exists a path $u$ such that $s(u)=e$ and $r(u)=f$. Since $0\cdot u=0$, the continuity of right shifts in $G(E)$ yields an open neighborhood $V\subset U$ of $0$ such that $V\cdot u\subset U$. Observe that the set $V\cap L_e$ is infinite and, by Lemma~\ref{lemmaeq}, $(V\cap L_e)\cdot u$ is an infinite subset of $L_f\cap U$. Hence we can apply lemmas~\ref{lemma3}--\ref{lemma7} to the vertex $f$ and obtain that the set $D_f\setminus U$ is finite. Since each non-zero $\mathcal{D}$-class contains a unique vertex and the graph $E$ has finitely many vertices, we conclude that $G(E)$ has finitely many $\mathcal{D}$-classes. Hence the set $G(E)\setminus U$ is finite.
\end{proof}

\section*{A generalization of Main Theorem}

Observe that Main Theorem remains true if the graph $E$ is a disjoint union of two graphs $E_1$ and $E_2$ such that the graph $E_1$ satisfies conditions of Main Theorem and the GIS $G(E_2)$ is finite.
 However, Main Theorem can not be generalized over the case when the graph $E$ is a disjoint union of two graphs $E_1$ and $E_2$ such that both semigroups $G(E_1)$ and $G(E_2)$ are infinite.

\begin{proposition}\label{prop1}
Let $E$ be a graph which is a disjoint union of two graphs $E_1$ and $E_2$ such that both semigroups $G(E_1)$ and $G(E_2)$ are infinite. Then there exists a topology $\tau$ on $G(E)$ such that $(G(E),\tau)$ is a locally compact, non-compact, non-discrete quasi-topological semigroup.
\end{proposition}

\begin{proof}
Assume that $E=E_1\sqcup E_2$ and both semigroups $G(E_1)$ and $G(E_2)$ are infinite. We introduce a topology $\tau$ on $G(E)$ in the following way:
each non-zero element $uv^{-1}$ is isolated in $G(E)$. An open neighborhood base of the point $0$ consists of cofinite subsets of $G(E_1)$ which contains point $0$.
Similar arguments as in Lemma~\ref{lemmacomp} imply the continuity of the inversion in $G(E)$.
To prove that $(G(E),\tau)$ is a semitopological semigroup we need to consider the following four cases:
\begin{itemize}
\item[1)] $uv^{-1}\cdot 0=0$, where $uv^{-1}\in G(E_1)$;
\item[2)] $0\cdot uv^{-1}=0$, where $uv^{-1}\in G(E_1)$;
\item[3)] $uv^{-1}\cdot 0=0$, where $uv^{-1}\in G(E_2)$;
\item[4)] $0\cdot uv^{-1}=0$, where $uv^{-1}\in G(E_2)$.
\end{itemize}
The continuity of left (resp. right) shifts in the first (resp. second) case follows from Lemma~\ref{lemmacomp}. The continuity of left and right shifts in cases three and four can be derived from the following equation:
$$uv^{-1}\cdot G(E_1)=G(E_1)\cdot uv^{-1}=0,\hbox{ where } uv^{-1}\in G(E_2).$$
\end{proof}
Theorem~\ref{Bard1} and Proposition~\ref{prop1} imply the following:
\begin{corollary}\label{prop2}
Let $G(E)$ be a GIS which satisfies the dichotomy of the Main Theorem, i.e., a locally compact shift-continuous topology on $G(E)$ is either compact or discrete. Then $G(E)$ satisfies the condition $(\star)$ and the graph $E$ cannot be represented as a union of two graphs $E_1$ and $E_2$ such that semigroups $G(E_1)$ and $G(E_2)$ are infinite.
\end{corollary}

The above Corollary leads us to the following question:
\begin{question*}
Is it true that a GIS $G(E)$ satisfies the dichotomy of Main Theorem iff $G(E)$ satisfies the condition~$(\star)$ and the graph $E$ cannot be represented as a disjoint union of two graphs $E_1$ and $E_2$ such that semigroups $G(E_1)$ and $G(E_2)$ are infinite?
\end{question*}

\section*{Acknowledgements}

The author acknowledges professor Taras Banakh for his comments and suggestions.

\end{document}